\documentclass[a4paper, 12pt, oneside]{article}
\pdfoutput=1 
\usepackage[T1]{fontenc}
\usepackage[british]{babel}
\usepackage{lmodern}
\usepackage{etoolbox}
\usepackage{amsmath, amsthm, amssymb, bm, mleftright}
\usepackage{booktabs, longtable, caption, subcaption, multirow}
\usepackage[space]{grffile}
\usepackage[margin=2.5cm, footskip=1cm]{geometry}
\usepackage{enumitem}
\setenumerate[1]{label=(\arabic*), ref=\arabic*}
\usepackage[protrusion=allmath]{microtype}
\usepackage{tikz}
\usetikzlibrary{cd, calc}
\usepackage{cellspace} 
\usepackage[table]{xcolor}
\colorlet{rowcolor}{gray!10}
\newcommand\midrulecolor{%
  \arrayrulecolor{white}\specialrule{\aboverulesep}{0pt}{0.15\baselineskip}%
  \arrayrulecolor{black}\specialrule{\lightrulewidth}{0pt}{0pt}%
  \arrayrulecolor{rowcolor}\specialrule{\belowrulesep}{0pt}{0.15\baselineskip}%
  \arrayrulecolor{black}%
}
\newcommand\bottomrulecolor{%
  \arrayrulecolor{rowcolor}\specialrule{\aboverulesep}{0pt}{0pt}%
  \arrayrulecolor{black}\specialrule{\heavyrulewidth}{0pt}{\belowbottomsep}%
}
\usepackage[colorlinks, citecolor=blue, linkcolor=blue, linktoc=section]{hyperref}

\renewcommand{\thefootnote}{\fnsymbol{footnote}}

\allowdisplaybreaks

\usepackage{titlesec}
\titleformat*{\section}{\large\bfseries}
\titleformat*{\subsection}{\normalsize\bfseries}
\newlength{\VerticalSpaceAfterParagraph}
\titlespacing*{\paragraph}{0pt}{\VerticalSpaceAfterParagraph}{1em}

\usepackage{titling}
\pretitle{\vspace{-\baselineskip}\begin{center}\Large\bfseries}
\posttitle{\end{center}\vspace{-0.25\baselineskip}}
\preauthor{\begin{center}}
\postauthor{\end{center}}
\predate{\begin{center}}
\postdate{\end{center}}

\setlist
  {
    topsep = 5.0pt plus 2.0pt minus 3.0pt,
    partopsep = 1.5pt plus 1.0pt minus 1.0pt,
    parsep = 2.5pt plus 1.25pt minus 0.5pt,
    itemsep = 0pt plus 1.25pt minus 0.5pt
  }

\theoremstyle{plain}
\newtheorem{theorem}{Theorem}
\newtheorem{proposition}[theorem]{Proposition}
\newtheorem{lemma}[theorem]{Lemma}
\newtheorem{corollary}[theorem]{Corollary}

\theoremstyle{definition}
\newtheorem{definition}[theorem]{Definition}
\newtheorem{notation}[theorem]{Notation}
\newtheorem{setting}[theorem]{Setting}

\theoremstyle{remark}
\newtheorem{remark}[theorem]{Remark}

\numberwithin{theorem}{section}
\numberwithin{equation}{section}

\newcommand\blankfootnote[1]
  {%
    \begin{NoHyper}%
      \renewcommand\thefootnote{}%
      \footnote{#1}%
      \addtocounter{footnote}{-1}%
    \end{NoHyper}%
  }

\newcommand*\abs[1]{\left\lvert #1 \right\rvert}

\newcommand\Thanks{Competing Interests: The first author was supported by the Simons Investigator Award HMS, National Science Fund of Bulgaria, National Scientific Program ``Excellent Research and People for the Development of European Science'' (VIHREN), Project No.~KP-06-DV-7 and this work is supported by SFB~195 No.~286237555 of DFG. The second author was partially supported by the Engineering and Physical Sciences Research Council (EPSRC) Grant EP/V048619/1 ``K\"ahler-Einstein metrics on Fano Manifolds'' and EP/V056689/1 ``The Calabi Problem for smooth Fano threefolds''.}

\title{Log canonical thresholds of high multiplicity reduced plane curves}
\author{Erik Paemurru, Nivedita Viswanathan}
\date{5th~September 2025}
\newcommand\keywords{Log Canonical Thresholds, Plane Curves, Singularities} 
\newcommand\subjclass{14B05, 14E15, 14H20, 14H50, 32S25}

\begin{document}

\maketitle

\begin{abstract}
We compute log canonical thresholds of reduced plane curves of degree $d$ at points of multiplicity $d-1$. As a consequence, we describe all possible values of log canonical threshold that are less than $2/(d-1)$ for reduced plane curves of degree $d$. In addition, we compute log canonical thresholds for all reduced plane curves of degree less than 6.
\blankfootnote{\textup{2020} \textit{Mathematics Subject Classification} \subjclass.}
\blankfootnote{\textit{Keywords}: \keywords.}
\blankfootnote{\Thanks{}}
\end{abstract}

\tableofcontents

\section{Introduction}

A classical question in singularity theory is to understand the complexity of the singularities of hypersurfaces. One invariant of hypersurface singularities is the \textit{log canonical threshold} which is determined by any resolution of the singularity. Log canonical thresholds appear in different problems pertaining to differential and algebraic geometry. For instance, the greatest root of the Bernstein polynomial is the negative of the log canonical threshold (see \cite[Theorem 10.6]{Kol97}). Another important application is in establishing the existence of K\"ahler-Einstein metrics on Fano varieties using Tian's criterion which is an asymptotic version of the log canonical threshold (\cite{Tian87}). In this regard, log canonical thresholds play a role in stability problems (see for example \cite{Zan22}).

Log canonical thresholds were first known as \textit{complex singularity exponents}. The log canonical threshold of a convergent power series $f \in \mathbb C\{x_1, \ldots, x_n\}$ can be defined as the supremum over real numbers $\lambda$ such that the integral of $1/\abs{f}^{2\lambda}$ converges around~$\bm 0$. The properties of log canonical thresholds have been studied with regards to the mixed Hodge structures of the vanishing cohomology (see for instance \cite{Stb77}, \cite{Var81}). This was then used in \cite{Stb77} to define an invariant called the \textit{spectrum} of an isolated hypersurface singularity, which was further generalised in \cite{Stb89} to any hypersurface singularity. By \cite[Section 1]{Stb89}, in the case of isolated hypersurface singularities, the spectral numbers of the singularity can be used to retrieve the Milnor number. By \cite[Section~4]{Var82} or \cite[Theorem~9.5]{Kol97}, the log canonical threshold of $f$ is $\min(1, \beta_\mathbb C(f))$, where $\beta_\mathbb C(f)$ is the complex singular index, and by \cite{Stb85}, the smallest spectral number in the spectral sequence is $\beta_\mathbb C(f) - 1$.

In this paper, we study log canonical thresholds of reduced plane curves $C \subset \mathbb{A}^2$ at a point $P \in C$, where \emph{reduced} means that the defining polynomial of the curve is not divisible by the square of any non-unit polynomial. The pair $(\mathbb{A}^2, C)$ is said to be \emph{log canonical} at $P$ if it has a \textit{log resolution} over $P$ such that locally the coefficients of all the prime divisors of the log pullback of $C$ are at most~$1$ (Definition~\ref{def:lct}). The \emph{log canonical threshold} of $C$ at $P$ is then given by
\[
\operatorname{lct}_P(\mathbb{A}^2, C) := \operatorname{sup} \mleft\{ \lambda \in \mathbb Q_{>0} \;\middle|\; (\mathbb{A}^2, \lambda C)\ \text{is log canonical at the point $P$} \mright\}.
\]

Log canonical threshold is roughly the reciprocal of multiplicity (\cite[Lemma~8.10.1]{Kol97} or \cite[Exercise~6.18 and Lemma~6.35]{KSC04}):
\begin{equation} \label{eqn:bounding lct using multiplicity}
\frac{1}{\operatorname{mult}_P(C)} \leq \operatorname{lct}_P(\mathbb{A}^2, C) \leq \frac{2}{\operatorname{mult}_P(C)}.
\end{equation}
Log canonical threshold is a finer invariant than multiplicity and Milnor number in the case of reduced plane curves of degree~5: there are 5 possible multiplicities (1--5), 16 possible Milnor numbers (0--14 and 16) and 24 possible log canonical thresholds (Table~\ref{tab:lct list}).

Equation~\eqref{eqn:bounding lct using multiplicity} shows that a high multiplicity corresponds to a low log canonical threshold. To put it more sharply, by \cite[Theorem~4.1]{Che01}, the least log canonical threshold of a reduced plane curve of degree $d$ is~$2/d$, which happens precisely when the multiplicity at the point is~$d$, implying that the curve is the union of $d$ lines. By \cite{Che17} or \cite[Proposition~4.5]{Vis20}, if $\operatorname{mult}_P(C) \leq d-2$, then $\operatorname{lct}_P(\mathbb A^2, C) \geq 2/(d-1)$. Therefore, the lowest log canonical thresholds, meaning the values between $2/d$ and~$2/(d-1)$, happen when the multiplicity at the point is high, meaning at least~$d-1$.

Our main result is giving a simple formula for the log canonical threshold of a reduced plane curve of degree~$d$ at a point on the curve of multiplicity~$d-1$:

\begin{theorem}[= Theorem~\ref{thm:lct of a curve}] \label{thm:maintheorem}
Let $C$ be a reduced plane curve of degree $d$ and $P$ a point of $C$ of multiplicity~$d-1$. Let $C^1$ be the strict transform of $C$ under the blowup of the plane $\mathbb{A}^2$ at~$P$ and let $E$ be the exceptional divisor. Then
\[
\operatorname{lct}_P(\mathbb{A}^2, C) < \frac{2}{d-1} \iff \exists Q\in C^1\colon \operatorname{mult}_Q(C^1 \cdot E) > \frac{d-1}{2}.
\]
If the point $Q$ exists, then it is unique. In this case,
\[
\operatorname{lct}_P(\mathbb{A}^2, C) = \mleft\{\begin{aligned}
  & \frac{2 \cdot \operatorname{mult}_Q(C^1 \cdot E) - 1}{d \cdot (\operatorname{mult}_Q(C^1 \cdot E) - 1) + 1} && \begin{aligned}
    \text{if $L_Q$ is an irreducible}\\
    \text{component of~$C$,}
  \end{aligned}\\
  & \frac{2 \cdot \operatorname{mult}_Q(C^1 \cdot E) + 1}{d \cdot \operatorname{mult}_Q(C^1 \cdot E)} && \text{otherwise},
\end{aligned}\mright.
\]
where $L_Q$ is the unique line on the affine plane containing $P$ such that its strict transform contains~$Q$.
\end{theorem}

It was proved in \cite[Theorem~1.10]{Che17} that for $d \geq 4$, the five smallest log canonical thresholds are
\[
\mleft\{ \frac{2}{d}, \frac{2d-3}{(d-1)^2}, \frac{2d-1}{d(d-1)}, \frac{2d-5}{d^2-3d+1}, \frac{2d-3}{d(d-2)} \mright\},
\]
and in \cite[Theorem~1.8]{Vis20}, that for $d \geq 5$ the sixth smallest log canonical threshold is $\frac{2d - 7}{d^2 - 4d + 1}$. A keen eye will notice a pattern in the six rational numbers above, namely that they contain two simple subsequences. We show that these subsequences can be extended. Moreover we describe all log canonical thresholds at points of multiplicity $d-1$:

\newcommand\CorollaryText[1]{
  Let $\Lambda_{d, d-1}$ denote the set of log canonical thresholds of pairs $(\mathbb A^2, C)$ at a point of multiplicity $d-1$ of a reduced plane curve $C$ of degree~$d$. Then for every $d \geq 3$,
  \begin{equation} #1
  \begin{aligned}
    \Lambda_{d, d-1} = \mleft\{ \frac{2}{d-1} \mright\} & \cup \mleft\{ \dfrac{2k + 1}{kd + 1} \;\middle|\; k \in \Bigl\{\bigl\lfloor\frac{d-1}{2}\bigr\rfloor, \ldots, d-2\Bigr\} \mright\}\\
    & \cup \mleft\{ \dfrac{2k + 1}{kd} \;\middle|\; k \in \Bigl\{\bigl\lfloor\frac{d+1}{2}\bigr\rfloor, \ldots, d-1\Bigr\} \mright\},
  \end{aligned}
  \end{equation}
  where $\lfloor x\rfloor$ denotes the greatest integer not greater than~$x$.
}

\begin{corollary}[= Corollary~\ref{cor:possible lct's}] \label{thm:corollary in intro}
\CorollaryText{\nonumber}
\end{corollary}

Lastly, we concentrate on low degree curves. Singularities of low degree plane curves have been intensively studied from various points of view. We fill a gap in this direction by computing the log canonical thresholds for all reduced plane curves of degree at most~$5$. By \cite{Var82}, log canonical threshold is constant in $\mu$-constant strata (Definition~\ref{def:normal forms}), meaning that all the power series in a connected component of set of the power series with given Milnor number have the same log canonical threshold at the origin. Therefore, we use the existing classification lists of singularities to compute log canonical thresholds.

\begin{table}[h!]
\centering
\caption{Log canonical thresholds of reduced plane curves\label{tab:lct list}}
\begin{tabular}{cSc}
  \toprule
  degree & $\operatorname{lct}_P(\mathbb{A}^n, C)$\\
  \midrulecolor
  \rowcolor{rowcolor}
  $1$ & $1$\\
  $2$ & $1$\\
  \rowcolor{rowcolor}
  $3$ & $1,~ \frac{5}{6},~ \frac{3}{4},~ \frac{2}{3}$\\
  $4$ & $1,~ \frac{5}{6},~ \frac{3}{4},~ \frac{7}{10},~ \frac{2}{3},~ \frac{9}{14},~ \frac{5}{8},~ \frac{3}{5},~ \frac{7}{12},~ \frac{5}{9},~ \frac{1}{2}$\\
  \rowcolor{rowcolor}
  $5$ & $\begin{gathered}
    1,~ \tfrac{5}{6},~ \tfrac{3}{4},~ \tfrac{7}{10},~ \tfrac{2}{3},~ \tfrac{9}{14},~ \tfrac{5}{8},~ \tfrac{11}{18},~ \tfrac{3}{5},~ \tfrac{13}{22},~ \tfrac{7}{12},~ \tfrac{15}{26},\\
    \tfrac{4}{7},~ \tfrac{9}{16},~ \tfrac{5}{9},~ \tfrac{11}{20},~ \tfrac{6}{11},~ \tfrac{8}{15},~ \tfrac{1}{2},~ \tfrac{7}{15},~ \tfrac{5}{11},~ \tfrac{9}{20},~ \tfrac{7}{16},~ \tfrac{2}{5}
  \end{gathered}$\\
  \bottomrulecolor
\end{tabular}
\end{table}

Table~\ref{tab:lct list} gives an exhaustive list of all possible log canonical thresholds of reduced plane curves $C$ of degree at most $5$ at a given point $P$.

The paper is organised as follows. In section~\ref{sec:prelim}, we set up the preliminary definitions and results that are needed to describe log canonical thresholds of reduced plane curves, such as the notion of power series and log resolutions. Section~\ref{sec:high multiplicity curves} is devoted towards proving Theorem~\ref{thm:maintheorem}(= Theorem~\ref{thm:lct of a curve}). In addition to this, we also provide the complete list of all possible values of log canonical thresholds that a pair $(\mathbb{A}^2,C)$ can take at a point $p\in C$ of multiplicity $d-1$ in a curve $C$ of degree $d$. This is Corollary~\ref{cor:possible lct's}.

Section~\ref{sec:low degree curves} contains Table~\ref{tab:low degree singularities} which lists all the singularities that reduced plane curves of degree at most $5$ can have and Table~\ref{tab:normal form notation} which lists the normal forms and the log canonical thresholds for each singularity. It also contains the proofs for why the tables give an exhaustive list.

\section{Preliminaries} \label{sec:prelim}

\begin{notation}
To avoid possible misunderstanding, we list some of the standard notation we use.
\begin{enumerate}
\item Variety --- an integral separated scheme of finite type over the complex numbers~$\mathbb C$.
\item Curve --- a reduced separated scheme of finite type over~$\mathbb C$ of pure dimension~$1$.
\item $\mathbb C\{x_1, \ldots, x_n\}$ --- the $\mathbb C$-algebra of power series in variables $x_1, \ldots, x_n$ that are absolutely convergent in a neighbourhood of the origin.
\item $(\mathbb{V}(f), \bm 0)$ --- the (possibly nonreduced) complex space subgerm of $(\mathbb C^n, \bm 0)$ defined by~$f \in \mathbb C\{x_1, \ldots, x_n\}$.
\item $f$ is square-free --- no square of a non-unit in $\mathbb C[x_1, \ldots, x_n]$ divides $f$.
\item Plane curve of degree~$d$ --- a scheme which is isomorphic to an open dense subscheme of $\operatorname{Proj} \mathbb C[x, y, z] / (f)$ for a square-free polynomial $f \in \mathbb C[x, y, z]$ homogeneous of degree~$d$, where $d$ is a positive integer.
\end{enumerate}
Given a positive integer~$n$, a nonzero convergent power series $f \in \mathbb C\{x_1, \ldots, x_n\}$ and positive rational numbers $(w_1, \ldots, w_n)$ called \emph{weights} corresponding to the variables $x_1, \ldots, x_n$, we have the following notation.
\begin{enumerate}[resume]
\item $\operatorname{wt}(f)$ --- the \textbf{weight} of~$f$, defined by
\[
\operatorname{wt}(f) := \min \mleft\{ i_1 w_1 + \ldots + i_n w_n \;\middle|\; \begin{aligned}
  & \text{$i_1, \ldots, i_n \in \mathbb Z_{\geq0}$, the coefficient}\\
  & \text{of $x_1^{i_1} \cdot \ldots \cdot x_n^{i_n}$ in $f$ is non-zero}
\end{aligned} \mright\},
\]
\item $\operatorname{mult}(f)$ ---  the multiplicity of~$f$, defined to be the weight of $f$ with respect to the weights $(1, \ldots, 1)$,
\item $f$ is quasihomogeneous --- all the monomials with a non-zero coefficient in $f$ have the same weight,
\item $f$ is semiquasihomogeneous --- the nonzero quasihomogeneous subpolynomial of $f$ that is of \textit{least weight} (that is, the sum of all the monomials of weight $\operatorname{wt}(f)$ together with their coefficients in~$f$) defines a smooth germ or an isolated singularity at the origin.
\end{enumerate}
Given curves $C$ and $C'$ containing a closed point~$P$, we use the notation below.
\begin{enumerate}[resume]
\item $\operatorname{mult}_P(C)$ --- the multiplicity of $C$ at $P$ is defined to be $\operatorname{mult}(f)$, where $f$ is any convergent power series in $\mathbb C\{x, y\}$ such that the complex space germ $(C^{\mathrm{an}}, P)$ is isomorphic to the complex space subgerm $(\mathbb V(f), \bm 0)$ of $(\mathbb C^2, \bm 0)$, where $C^{\mathrm{an}}$ denotes the analytification of~$C$,
\item $\operatorname{mult}_P(C \cdot C')$ --- the intersection multiplicity of $C$ and $C'$ along~$P$ is defined in \cite[Definition~7.1]{Ful98} and can be computed using \cite[Example~7.1.10(b)]{Ful98}.
\end{enumerate}
\end{notation}

\subsection{Log resolutions}

Below, we give the technical definitions of log resolution and the log pullback $D'$ of a divisor~$D$. The characterising property of log pullback in the setting of Definition~\ref{def:log pullback} is the linear equivalence
\[
K_{S'} + D' \sim \varphi^*(K_S + D).
\]

\begin{definition}[{\cite[Notation~0.4]{KM98}}] \label{def:log pullback}
Let $S$ be a smooth variety. A \textbf{$\mathbb Q$-divisor} on $S$ is a formal $\mathbb Q$-linear combination $\sum \lambda_i D_i$ of prime divisors $D_i$ where $\lambda_i \in \mathbb Q$. An effective $\mathbb Z$-divisor $\sum \lambda_i D_i$ is \textbf{snc} if all the prime divisors $D_i$ are smooth and around every point of~$S$, $\sum \lambda_i D_i$ is locally analytically given by $V(x_1^{a_1} \cdot \ldots \cdot x_n^{a_n})$ in $\mathbb C^n$ where $(a_1, \ldots, a_n) \in \mathbb Z_{\geq0}^n$.

Let $D$ be a $\mathbb Q$-divisor on a smooth variety~$S$. A \textbf{log resolution of $(S, D)$ over~$P$} is a proper birational morphism $\pi\colon S' \to S$ from a scheme $S'$ such that there exists an open neighbourhood $U \subseteq S$ of~$P$ such that $\pi^{-1}U$ is a smooth variety, the exceptional locus $E$ of $\mleft.\pi\mright|_{\pi^{-1}U}$ is of pure codimension $1$ and $E \cup \mleft.\pi\mright|_{\pi^{-1}U}^{-1}(\operatorname{Supp}(D \cap U))$ is an snc divisor of~$\pi^{-1}U$.

For any proper birational morphism $\varphi\colon S' \to S$ from a smooth variety~$S'$, the \textbf{relative canonical divisor} of~$\varphi$, denoted~$K_\varphi$, is the unique $\mathbb Q$-divisor that is linearly equivalent to $\varphi^*(K_S) - K_{S'}$ and supported on the exceptional locus of~$\varphi$, where $K_S$ and $K_{S'}$ are the canonical classes of respectively $S$ and~$S'$. The \textbf{log pullback} of $D$ with respect to $\varphi$ is the $\mathbb Q$-divisor $D' = K_\varphi + \varphi^* D$ on~$S'$.
\end{definition}

Now we are ready to define the log canonical threshold.

\begin{definition}[{\cite[Definition 3.5]{Kol97} or \cite[Definition~2.34]{KM98}}] \label{def:lct}
Let $D$ be a $\mathbb Q$-divisor on a smooth variety $S$ and let $P \in S$ be a point. The pair $(S, D)$ is \textbf{log canonical at~$P$} if we can restrict $(S, D)$ to an open neighbourhood of $P$ such that there exists a log resolution with all the coefficients of the prime divisors in the log pullback of $D$ at most~$1$. The \textbf{log canonical threshold of $(S, D)$ at~$P$} is
\[
\operatorname{lct}_P(S, D) := \operatorname{sup}\bigl\{ \lambda \in \mathbb Q_{>0} \;\big|\; (S, \lambda D) \text{ is log canonical at~$P$} \bigr\}.
\]
\end{definition}

Note that log canonical threshold at a closed point is a local analytic invariant ({\cite[Proposition~4-4-4]{Mat02}}). The following lemma is used to prove Theorem~\ref{thm:maintheorem} and to compute the log canonical thresholds of the polynomials in Table~\ref{tab:normal form notation}.

\begin{lemma}[{\cite[Propositions 8.13 and~8.14 and Remark~8.14.1]{Kol97} or \cite[Proposition~2.1]{Kuw99}}] \label{lem:computing lct}
Let $f \in \mathbb C\{x_1, \ldots, x_n\}$. Assign positive rational weights $\bm w = (w_1, \ldots, w_n)$ to the variables. Let $f_w$ denote the weighted homogeneous leading term of~$f$. Define $b := \sum_i w_i / \operatorname{wt}(f)$. Considering $\mathbb C^n$, $\mathbb{V}(f)$ and $\mathbb{V}(f_{\bm w})$ as complex space germs around~$\bm 0$, we have $\operatorname{lct}_{\bm 0}(\mathbb C^n, \mathbb{V}(f)) \leq b$. Moreover, if the pair $(\mathbb C^n, b \mathbb{V}(f_w))$ is log canonical outside the origin, then $\operatorname{lct}_{\bm 0}(\mathbb C^n, \mathbb{V}(f)) = b$.
\end{lemma}

\subsection{Power series}

We use the standard definitions below in section~\ref{sec:low degree curves}.

\begin{definition}[{\cite[Definitions I.1.1, I.1.47 and~I.2.1]{GLS07}}]
Let $n$ be a positive integer. We denote by
\begin{itemize}
\item $\mu(f)$ --- the \textbf{Milnor number} of~$f \in \mathbb C\{x_1, \ldots, x_n\}$, defined by
\[
\mu(f) = \dim_\mathbb C \frac{\mathbb C\{x_1, \ldots, x_n\}}{(\frac{\partial f}{\partial x_1}, \ldots, \frac{\partial f}{\partial x_n})}.
\]
\end{itemize}
\end{definition}

\begin{definition}[{\cite[Definition I.2.9]{GLS07} and \cite[Introduction to Part~II]{AGZV85}}]\label{def:righteq}
Let $n \leq m$ be positive integers. Let $f, g \in \mathbb C\{x_1, \ldots, x_n\}$ and $h \in \mathbb C\{y_1, \ldots, y_m\}$.
\begin{itemize}
    \item We say that $f$ and $g$ are \textbf{right equivalent} if there exists an automorphism $\Phi$ of $\mathbb C\{x_1, \ldots, x_n\}$ such that $\Phi(f) = g$. In simpler terms, $f$ and $g$ are right equivalent if they coincide up to local analytic coordinate changes.
    \item We say that $f$ and $h$ are \textbf{stably right equivalent} if there exists a non-negative integer $k$ and an isomorphism $\Psi\colon \mathbb C\{x_1, \ldots, x_{n+k}\} \to \mathbb C\{y_1, \ldots, y_{n+k}\}$ such that $\Psi(f + x_{n+1}^2 + \ldots + x_{n+k}^2) = h + y_{m+1}^2 + \ldots + y_{n+k}^2$.
\end{itemize}
\end{definition}

\begin{remark} \label{rem:stably equiv iff equiv}
Normal forms of singularities are often considered up to stable right equivalence. So, while the power series $x^3 + y^6 + z^2 + 2 xyz$ defines a surface singularity, it is stably right equivalent to the power series $x^2 y^2 + x^3 + y^6$ which defines a curve singularity. Due to this, there are several different notations for the same singularity class (Remark~\ref{rem:multiple notations}(\ref{itm:multiple notations})). Note that if the number of variables is the same, then two power series $f, g \in \mathbb C\{x_1, \ldots, x_n\}$ are stably right equivalent if and only if they are right equivalent, see \cite[Remark in Section~11.1]{AGZV85}.
\end{remark}

\section{High multiplicity curves} \label{sec:high multiplicity curves}

In this section, we classify log canonical thresholds at points of multiplicity $d-1$ for reduced plane curves of degree $d$. The notation we use is given in Setting~\ref{set:curve C}.

\begin{setting} \label{set:curve C}
Let $d \geq 3$ be an integer. Let $P$ be a point of a reduced affine plane curve $C$ of degree $d$ such that $\operatorname{mult}_P C = d-1$. Let $C^1$ be the strict transform of $C$ under the blowup of the affine plane along $P$ with exceptional divisor~$E_1$. For every point $Q \in C^1$ such that $\operatorname{mult}_Q(C^1 \cdot E_1) > 1$, let $L_Q$ be the line on the affine plane through $P$ such that its strict transform passes through~$Q$. For every such $Q \in C^1$, let $C_Q^1$ be the strict transform of the Zariski closure $C_Q$ of $C \setminus L_Q$. Define the positive integer $k_Q$ by
\[
k_Q := \operatorname{mult}_Q(C_Q^1 \cdot E_1)
\]
and define the positive rational number $l_Q$ by
\[
l_Q := \mleft\{\begin{aligned}
  \dfrac{2 k_Q + 1}{k_Q d + 1} & \quad \text{if $L_Q$ is an irreducible component of $C$},\\
  \dfrac{2 k_Q + 1}{k_Q d} & \quad \text{otherwise}.
\end{aligned}\mright.
\]
\end{setting}

By equation~\eqref{eqn:bounding lct using multiplicity}, $\operatorname{lct}_P(\mathbb{A}^2, C) \leq 2/(d-1)$. The main theorem of this section is as follows.

\begin{theorem} \label{thm:lct of a curve}
With the notations and setup as in Setting~\ref{set:curve C},
\[
\operatorname{lct}_P(\mathbb{A}^2, C) < \frac{2}{d-1} \iff \exists Q\in C^1\colon \operatorname{mult}_Q(C^1 \cdot E_1) > \frac{d-1}{2}.
\]
Moreover, if the point $Q$ exists, then it is unique and $\operatorname{lct}_P(\mathbb{A}^2, C) = l_Q$.
\end{theorem}

To prove Theorem~\ref{thm:lct of a curve}, we first describe $C$ using equations:

\begin{lemma} \label{lem:relating C and f}
We say that two triples $(C, P, Q)$ and $(C', P', Q')$ are \emph{isomorphic} if there exists an isomorphism $C \to C'$ of curves that takes the point $P$ to $P'$ and $Q$ to~$Q'$. Let $\operatorname{Spec} \mathbb C[x_1, y] \cong \operatorname{Spec} \mathbb C[x/y, y] \gets \mathbb C[x, y]$ be one of the affine opens of the blowup of $\mathbb A^2 := \operatorname{Spec} \mathbb C[x, y]$ at~$\bm 0$. Then, up to isomorphism, the triples $(C, P, Q)$ in Setting~\ref{set:curve C} are precisely given by
\[
\begin{aligned}
  C & = \mathbb V(f) \subseteq \mathbb A^2,\\
  P & = \bm 0 \in \mathbb A^2,\\
  Q & = \bm 0 \in \operatorname{Spec} \mathbb C[x_1, y],
\end{aligned}
\]
where $d \in \mathbb Z_{\geq3}$, $a_i, b_j \in \mathbb C$, $a_{k_Q} \neq 0$, $b_0 \neq 0$, $f$ is square-free and where one of the following holds
\begin{itemize}
\item $L_Q$ is an irreducible component of~$C$, $k_Q \in \{1, \ldots, d - 2\}$ and
\[
  f := x \mleft( \sum_{i\in \{k_Q,\, k_Q + 1,\, \ldots,\, d-2\}} a_i x^i y^{d-2-i} + \sum_{j \in \{0,\, 1,\, \ldots,\, d-1\}} b_j x^j y^{d-1-j} \mright),
\]
or
\item $L_Q$ is not an irreducible component of $C$, $k_Q \in \{2, \ldots, d - 1\}$ and
\[
  f := \sum_{i\in \{k_Q,\, k_Q + 1,\, \ldots,\, d-1\}} a_i x^i y^{d-1-i} + \sum_{j \in \{0,\, 1,\, \ldots,\, d\}} b_j x^j y^{d-j}.
\]
\end{itemize}
In both cases, $L_Q$ and $E_1 \cap \operatorname{Spec} \mathbb C[x_1, y]$ from Setting~\ref{set:curve C} correspond respectively to $\mathbb V(x) \subseteq \mathbb A^2$ and $\mathbb V(y) \subseteq \operatorname{Spec} \mathbb C[x_1, y]$.
\end{lemma}

\begin{proof}
We show how every $(C, P, Q)$ from Setting~\ref{set:curve C} is given by some $(\mathbb V(f), \bm 0, \bm 0)$. First, translate $P$ to~$\bm 0$ on~$\mathbb A^2$. Then use a linear invertible map on $\mathbb A^2$ fixing $\bm 0$ to move $Q$ to $\bm 0 \in \operatorname{Spec} \mathbb C[x_1, y]$. It follows that $L_Q = \mathbb V(x)$ and $E_1 \cap \operatorname{Spec} \mathbb C[x_1, y] = \mathbb V(y)$. The curve $C$ is given by $\mathbb V(g) \subseteq \mathbb A^2$ where $g = g_{d-1} + g_d \in \mathbb C[x, y]$, where $g_{d-1}$ and $g_d$ are respectively homogeneous of degrees $d-1$ and~$d$. The curve $C^1 \cap \operatorname{Spec} \mathbb C[x_1, y]$ is given by $g_{d-1}(x_1, 1) + y g_d(x_1, 1)$. We see that $C^1 \cdot E = \operatorname{mult} g_{d-1}(x_1, 1)$. This shows that $g$ is equal to $f$ for some choice of $a_i$ and~$b_j$.

Conversely, if $(C, P, Q)$ are given by some $(\mathbb V(f), \bm 0, \bm 0)$ as above, then $C$ is a reduced affine plane curve of degree~$d$, $P \in C$ a point of multiplicity $d-1$, $\operatorname{mult}_Q(C_Q^1 \cdot E_1) = k_Q$ and $\operatorname{mult}_Q(C^1 \cdot E_1) > 1$.
\end{proof}

\begin{proof}[Proof of Theorem~\ref{thm:lct of a curve}]
First, it is easy to compute that for any point~$Q$, $l_Q < 2/(d-1)$ if and only if $\operatorname{mult}_Q(C^1 \cdot E_1) > (d-1)/2$. If such a point $Q$ exists, then it is necessarily unique since
\[
\sum_{Q \in C^1 \cap E_1} \operatorname{mult}_Q(C^1 \cdot E_1) = C^1 \cdot E_1 = \operatorname{mult}_P(C) = d-1.
\]

Below, we construct explicit log resolutions of $(\mathbb A^2, C)$ over the point $P$ to then obtain the value of $\operatorname{lct}_P(\mathbb A^2, C)$. Let $\pi_1\colon S_1 \to \mathbb{A}^2$ be the blowup along~$P$. If for every point $Q \in C^1 \cap E_1$ we have that $\operatorname{mult}_Q(C^1 \cdot E_1) = 1$, then $\operatorname{lct}_P(\mathbb{A}^2, C) = 2/(d-1)$. Otherwise, let $Q$ be any point of $C^1$ such that $\operatorname{mult}_Q(C^1 \cdot E_1) > 1$.
 By Lemma~\ref{lem:relating C and f}, $C_Q^1 \cap \operatorname{Spec} \mathbb C[x_1, y]$ is given by
\begin{equation} \label{eqn:strict transform C1 L}
  f_1 := \sum_{i\in \{k_Q,\, k_Q + 1,\, \ldots,\, d-2\}} a_i x_1^i + y \sum_{j \in \{0,\, 1,\, \ldots,\, d-1\}} b_j x_1^j,
\end{equation}
if $L_Q$ is an irreducible component of~$C$ and by
\begin{equation} \label{eqn:strict transform C1 no L}
  f_1 := \sum_{i\in \{k_Q,\, k_Q + 1,\, \ldots,\, d-1\}} a_i x_1^i + y \sum_{j \in \{0,\, 1,\, \ldots,\, d\}} b_j x_1^j.
\end{equation}
otherwise. Every irreducible component of $C$ that is not a line has degree strictly greater than its multiplicity at~$P$. Therefore, the curve $C$ has exactly one irreducible component which is not a line. Therefore, $C_Q^1$ has exactly one irreducible component passing through~$Q$. The strict transform $L_Q^1$ of $L_Q$ under $\pi_1$ is given on $\operatorname{Spec} \mathbb C[x_1, y]$ by~$\mathbb V(x_1)$. Using equations \eqref{eqn:strict transform C1 L} and~\eqref{eqn:strict transform C1 no L},
we see that $L_Q^1$ and $C_Q^1$ intersect in exactly one point, namely~$Q$, and the intersection is transversal. In particular, $Q$ is a smooth point of~$C_Q^1$.

For every rational number~$\lambda > 1/(d-1)$, let $D_\lambda$ denote the effective $\mathbb Q$-divisor
\[
D_\lambda := \lambda C^1 + (\lambda(d-1) - 1) E_1
\]
on $S_1$. Then $D_\lambda$ is the log pullback of $\lambda C$ under~$\pi_1$.

We describe a log resolution $\pi_2 \circ \ldots \circ \pi_{k_Q + 1}$ over~$Q$ of $(S_1, D_\lambda)$.
Let $\pi_2\colon S_2 \to S_1$ be the blowup along~$Q$. For every $r \in \{2, \ldots, k_Q\}$, let $\pi_{r+1}\colon S_{r+1} \to S_r$ be the blowup along the point $Q_r := \mathbb{V}(x_1, y_r)$ of the affine open $\operatorname{Spec} \mathbb C[x_1, y_r] \cong \operatorname{Spec} \mathbb C[x_1, y_{r-1}/x_1]$ of~$S_r$, where $y_1 := y$. Let $C_Q^r$ denote the strict transform of $C_Q$ under $\pi_1 \circ \ldots \circ \pi_r$. We see using equations \eqref{eqn:strict transform C1 L} and \eqref{eqn:strict transform C1 no L} that $C_Q^r \cap \operatorname{Spec} \mathbb C[x_1, y_r]$ is given by
\begin{equation} \label{eqn:strict transform Ci L}
  f_r := \sum_{i\in \{k_Q,\, k_Q + 1,\, \ldots,\, d-2\}} a_i x_1^{i-r+1} + y_r \sum_{j \in \{0,\, 1,\, \ldots,\, d-1\}} b_j x_1^j,
\end{equation}
if $L_Q$ is an irreducible component of~$C$ and by
\begin{equation} \label{eqn:strict transform Ci no L}
  f_r := \sum_{i\in \{k_Q,\, k_Q + 1,\, \ldots,\, d-1\}} a_i x_1^{i-r+1} + y_r \sum_{j \in \{0,\, 1,\, \ldots,\, d\}} b_j x_1^j.
\end{equation}
otherwise. Let $E_r$ be the exceptional divisor of~$\pi_r$ and let $E_i^r$ be the strict transform of the exceptional divisor of $\pi_i$ under $\pi_i \circ \ldots \circ \pi_r$. We have
\[
\begin{aligned}
  E_r \cap \operatorname{Spec} \mathbb C[x_1, y_r] & = \mathbb{V}(x_1),\\
  E_1^r \cap \operatorname{Spec} \mathbb C[x_1, y_r] & = \mathbb{V}(y_r).
\end{aligned}
\]
We find that $E_r$ and $E_1^r$ intersect in exactly one point, namely~$Q_r$, and the intersection is transversal. Moreover, $Q_r \notin E_2^r \cup E_3^r \cup \ldots \cup E_{r-1}^r$ and therefore, $\sum_{i \in \{1, \ldots, r\}} E_i^r$ is snc. Using equations \eqref{eqn:strict transform Ci L} and \eqref{eqn:strict transform Ci no L}, we see that $C_Q^r$ and $E_r$ intersect in exactly one point, namely~$Q_r$, and the intersection is transversal. Therefore, $(E_2^r \cup E_3^r \cup \ldots \cup E_{r-1}^r) \cap C_Q^r$ is empty.  From equations \eqref{eqn:strict transform Ci L} and \eqref{eqn:strict transform Ci no L}, we find
\[
\operatorname{mult}_{Q_r}(C_Q^r \cdot E_1^r) = k_Q - r + 1.
\]
The varieties $C_Q^{k_Q}$, $E_{k_Q}$ and $E_1^{k_Q}$ have pairwise transversal intersections at~$Q_{k_Q}$. Therefore, $\pi_2 \circ \ldots \circ \pi_{k_Q + 1}$ is a log resolution of $(S_1, D_\lambda)$ over~$Q$.

Let $D_\lambda^{k_Q + 1}$ denote the strict transform of $D_\lambda$ under $\pi_2 \circ \ldots \circ \pi_{k_Q + 1}$. The log pullback of $D_\lambda$ under the composition $\pi_2 \circ \ldots \circ \pi_{k_Q + 1}$ is given by
\[
D_\lambda^{k_Q + 1} + \sum_{j \in \{1, \ldots, k_Q\}} (\lambda(jd + 1) - 2j) E_{j+1}^{k_Q + 1}
\]
if $L_Q$ is an irreducible component of $C$ and
\[
D_\lambda^{k_Q + 1} + \sum_{j \in \{1, \ldots, k_Q\}} (\lambda jd - 2j) E_{j+1}^{k_Q + 1}
\]
otherwise.

We have the following equivalences:
\[
\begin{aligned}
  \lambda(d-1) - 1 & \leq 1 & \iff~\, \lambda & \leq \frac{2}{d - 1},\\
  \lambda(jd + 1) - 2j & \leq 1 & \iff~\, \lambda & \leq \frac{2j + 1}{jd + 1},\\
  \lambda jd - 2j & \leq 1 & \iff~\, \lambda & \leq \frac{2j + 1}{jd}.
\end{aligned}
\]
Note that
\[
l_Q = \mleft\{\begin{aligned}
  \min \mleft\{ \frac{2j+1}{jd+1} \;\middle|\; j \in \{1, \ldots, k_Q\} \mright\} & \quad \begin{aligned}
    \text{if $L_Q$ is an irreducible}\\
    \text{component of~$C$,}
  \end{aligned}\\
  \min \mleft\{ \frac{2j+1}{jd} \;\middle|\; j \in \{1, \ldots, k_Q\} \mright\} & \quad \text{otherwise}.
\end{aligned}\mright.
\]

Let $\pi$ be the composition of the blowup $\pi_1$ with the $\sum_Q k_Q$ blowups above, where the sum is over points $Q \in C^1$ such that $\operatorname{mult}_Q(C^1 \cdot E_1) > 1$. Then $\pi$ is a log resolution of $(\mathbb{A}^2, C)$ over~$P$. The log canonical threshold of $(\mathbb{A}^2, C)$ at~$P$ is by definition the greatest positive rational number $\lambda$ such that all the coefficients of the prime divisors in the log pullback of $\lambda C$ with respect to $\pi$ are at most~$1$. Therefore,
\[
\operatorname{lct}_P(\mathbb{A}^2, C) = \min \mleft( \mleft\{\dfrac{2}{d-1}\mright\} \cup \mleft\{ l_Q \;\middle|\; Q \in C^1,~ \operatorname{mult}_Q(C^1 \cdot E_1) > 1 \mright\} \mright).
\qedhere
\]
\end{proof}

\begin{remark}
Theorem~\ref{thm:lct of a curve} can also be proved using a sequence of blowups at points and inversion of adjunction similarly to the proofs of \cite[Theorem~1.10]{Che17} and \cite[Theorem~1.8]{Vis20}, or using the Newton polyhedron of the defining polynomial, see \cite[Theorem~3.10]{Pae24Reading}.

The method adopted in this paper works by giving very explicit equations to the curves under considerations, thus providing an intuition to the patterns of degrees of curves observed upon different blow ups and the computations of log canonical thresholds at the points of blow up.
\end{remark}

\begin{corollary} \label{cor:possible lct's}
\CorollaryText{\label{eqn:Lambda d d-1}}
\end{corollary}

\begin{proof}
If $f \in \mathbb C[x, y]$ is a multiplicity $d-1$ polynomial such that its homogeneous degree $d-1$ part is square-free, then $\operatorname{lct}_{\bm 0}(\mathbb A^2, V(f)) = 2/(d-1)$. This shows that $2/(d-1) \in \Lambda_{d, d-1}$.

In Setting~\ref{set:curve C} by Theorem~\ref{thm:lct of a curve}, we have
\begin{equation} \label{eqn:lct less than 2 over d-1}
\operatorname{lct}_P(\mathbb A^2, C) \leq 2/(d-1)
\end{equation}
and the strict inequality holds in equation~\eqref{eqn:lct less than 2 over d-1} if and only if there exists a point $Q \in C^1$ such that $\operatorname{mult}_Q(C^1 \cdot E_1) > 1$ one of the following holds:
\begin{enumerate}[label=(\theequation), ref=\theequation, itemindent=1.5em]
\let\olditem\item
\renewcommand\item{\stepcounter{equation}\olditem}
\item \label{ite:k lower bound L irred comp} $k_Q \geq \lfloor(d-1)/2\rfloor$ and $L_Q$ is an irreducible component of $C$, or
\item \label{ite:k lower bound otherwise} $k_Q \geq \lfloor(d+1)/2\rfloor$ and $L_Q$ is not an irreducible component of~$C$.
\end{enumerate}
Moreover, if (\ref{ite:k lower bound L irred comp}) or (\ref{ite:k lower bound otherwise}) holds for some $Q \in C^1$, then $\operatorname{lct}_P(\mathbb A^2, C) = l_Q$. Denoting the right-hand side of equation~\eqref{eqn:Lambda d d-1} by $\mathrm{RHS}$, we find $\Lambda_{d, d-1} \subseteq \mathrm{RHS}$. On the other hand, it is easy to construct examples of reduced plane curves satisfying (\ref{ite:k lower bound L irred comp}) or (\ref{ite:k lower bound otherwise}) using Lemma~\ref{lem:relating C and f}. This proves that $\mathrm{RHS} \subseteq \Lambda_{d, d-1}$.
\end{proof}

\begin{remark} \label{thm:lct's are disjoint}
The sets
\begin{equation} \label{eqn:reducible}
\mleft\{ \dfrac{2k + 1}{kd + 1} \;\middle|\; k \in \Bigl\{\bigl\lfloor\frac{d-1}{2}\bigr\rfloor, \ldots, d-2\Bigr\} \mright\}
\end{equation}
and
\[
\mleft\{ \dfrac{2k + 1}{kd} \;\middle|\; k \in \Bigl\{\bigl\lfloor\frac{d+1}{2}\bigr\rfloor, \ldots, d-1\Bigr\} \mright\}
\]
that appear in Corollary~\ref{cor:possible lct's} are disjoint for every integer $d \geq 3$. One implication is that if the log canonical threshold of a plane curve of degree $d \geq 3$ is in the set \eqref{eqn:reducible} above, then the curve is reducible.
\end{remark}

\section{Low degree curves} \label{sec:low degree curves}
In this section, we give an explicit list of all possible values of log canonical thresholds for lower degree curves. For this, we first present the list of all possible singularities that a curve $C$ of degree $d\leq 5$ can contain, in section~\ref{subsec:singularities of low degree curves} and for each of this singularity type, the corresponding normal form is presented in section~\ref{subsec: Measuring singularities using their normal forms}, along with the log canonical thresholds of the pair $(\mathbb{A}^2,C)$ at $O\in C$ with the singularity at the origin $O\in C$.

\subsection{Singularities of low degree curves}\label{subsec:singularities of low degree curves}

\addtolength{\jot}{-0.25\baselineskip}

\begin{table}[h!]
\centering
\caption{Singularities of reduced plane curves of given degree\label{tab:low degree singularities}}
\begin{tabular}{cSc}
  \toprule
  Degree $d$ & Possible singularities\\
  \midrulecolor
  \rowcolor{rowcolor}
  $2$ & $A_1$\\
  $3$ & $A_1, A_2, A_3, D_4$\\
  \rowcolor{rowcolor}
  $4$ & $A_1, \ldots, A_{7},~ D_4, D_5, D_6, E_6, E_7, T_{2,4,4}$\\
  $5$ & $
    \begin{gathered}
      A_1, \ldots, A_{12},~ D_4, \ldots, D_{12},~ E_6, E_7, E_8,~ T_{2,3,6}, \ldots, T_{2,3,10},\\
      T_{2,4,4}, T_{2,4,5}, T_{2,4,6}, T_{2,5,5}, T_{2,5,6}, T_{2,6,6}, Z_{11}, Z_{12}, W_{12}, W_{13}, N_{16}
    \end{gathered}
  $\\
  \bottomrule
\end{tabular}
\end{table}

\addtolength{\jot}{0.25\baselineskip}

\begin{proposition}
Every singularity of every reduced affine plane curve of degree $d \leq 5$ is of one of the types given in row $d$ of Table~\ref{tab:low degree singularities}. Conversely, for every singularity type given in row $d$ of Table~\ref{tab:low degree singularities}, there exists a square-free degree $d$ polynomial $f \in \mathbb C[x, y]$ such that its right equivalence class has non-empty intersection with the $\mu$-constant stratum of the normal form given in Table~\ref{tab:normal form notation} of the singularity.
\end{proposition}

\begin{proof}
The normal forms for $d \leq 5$ are well-known. The lists for $d = 4$ are given in \cite[Section~2]{WW09} and the lists for $d = 5$ are given in \cite[Section~3]{WW09} or \cite{Wal96}. For every normal form~$\Phi$ in the degree $d = 4$ row of Table~\ref{tab:low degree singularities}, \cite[Appendix~A]{WW09} contains an example of a quartic polynomial $f$ such that $f$ belongs to the $\mu$-constant stratum of~$\Phi$. For every normal form~$\Phi$ in the degree $d = 5$ row of Table~\ref{tab:low degree singularities}, \cite[Section~3]{WW09} and its erratum describe all quintic polynomials $f$ such that $f$ belongs to the $\mu$-constant stratum of~$\Phi$.
\end{proof}

\begin{remark}
\cite{WW09} actually proves more, namely the classification of real normal forms. Considered as complex normal forms, the normal forms with a star symbol are the same as without the star, for example $A_k = A_k^*$, $D_k = D_k^*$, $X_9 = X_9^* = X_9^{**}$, etc.
\end{remark}

\subsection{Measuring singularities using their normal forms}\label{subsec: Measuring singularities using their normal forms}

We define normal forms for the $\mu$-constant stratum. Note that in \cite[Section~15.0]{AGZV85}, normal forms are defined more generally for any class of singularities, not only the \mbox{$\mu$-c}onstant stratum, and the image of $\Phi$ is the whole polynomial ring $\mathbb C[x_1, \ldots, x_n]$, not a jet space. The reason it suffices to consider a jet space here is that a convergent power series $f$ of finite Milnor number $\mu(f)$ is $(\mu(f) + 1)$-determined, see \cite[Corollary~I.2.24]{GLS07}.

\begin{definition}[{\cite[Section~15.0]{AGZV85}}] \label{def:normal forms}
Let $n$ be a positive integer, let $m$ be a \mbox{non-n}egative integer and let $f \in \mathbb C\{x_1, \ldots, x_n\}$ have finite Milnor number~$\mu(f)$. The \textbf{$m$-jet} of $f$ is the sum over $k \in \{0, \ldots, m\}$ of the homogeneous degree $k$ parts of~$f$. The \textbf{$m$\mbox{-}jet space}, denoted $\mathbb C[x_1, \ldots, x_n]_{\leq m}$, is the $\mathbb C$-vector space of polynomials in $\mathbb C[x_1, \ldots, x_n]$ of degree at most~$m$. As a $\binom{n+m-1}{m}$-dimensional vector space over~$\mathbb C$, the $m$-jet space has a natural structure of a smooth complex space.
The \textbf{$\mu$-constant stratum of $f$} is the connected component of the $(\mu(f)+1)$-jet space of polynomials with Milnor number $\mu(f)$ which contains the $(\mu(f)+1)$-jet of~$f$. A \textbf{normal form of $f$} is a holomorphic map $\mathbb C^m \to \mathbb C[x_1, \ldots, x_n]_{\leq \mu(f)+1}$ such that all of the following hold:
\begin{enumerate}[label=(\arabic*), ref=\arabic*]
\item $\Phi(\mathbb C^m)$ intersects the right equivalence class of every polynomial in the $\mu$-constant stratum of~$f$,
\item the inverse image under $\Phi$ of every right equivalence class in $\Phi(\mathbb C^m)$ is finite, and
\item the inverse image under $\Phi$ of the complement of the $\mu$-constant stratum of $f$ is contained in a closed analytic proper subset of~$\mathbb C^m$.
\end{enumerate}
A \textbf{normal form} is a holomorphic map $\mathbb C^m \to \mathbb C[x_1, \ldots, x_n]_{\leq k}$, where $k$ is a positive integer, which is a normal form of some polynomial in its image.
A \textbf{polynomial normal form} is a normal form $\Phi$ such that $\Phi$ is a polynomial map.
The \textbf{$\mu$-constant stratum of a normal form $\Phi$} is the $\mu$-constant stratum of a polynomial $f$ such that $\Phi$ is a normal form of~$f$.
\end{definition}

All of the normal forms below are polynomial normal forms. Table~\ref{tab:normal form notation} contains the notation from \cite[Sections~15]{AGZV85} (or \cite[Section~13]{Arn75}) for the normal forms that we use. In Table~\ref{tab:normal form notation}, $a, b$ and $c$ are complex numbers, $k, q$ and $r$ are positive integers, $\operatorname{mult}$ stands for multiplicity, $\mu$ for Milnor number, $\operatorname{lct}$ for log canonical threshold, \emph{restrictions} describes the domain of the indices and \emph{$\mu$-constant stratum} describes the intersection of the image and the $\mu$-constant stratum of the normal form.

\newlength\MuConstantWidth
\settowidth{\MuConstantWidth}{$\mu$-constant}
\begin{table}[h!]
\centering
\caption{Notation for normal forms\label{tab:normal form notation}}
\begin{tabular}{*{4}{l}ccScc}
  \toprule
  Symbol & Indices & Normal form & \begin{minipage}{\MuConstantWidth}$\mu$-constant\\stratum\vspace{0.2em}\end{minipage} & $\operatorname{mult}$ & $\mu$ & $\mathrm{lct}$\\
  \midrulecolor
  \rowcolor{rowcolor}
  $A_k$ & $k \geq 1$ & $x^2 + y^{k+1}$ && $2$ & $k$ & $\frac{k+3}{2(k+1)}$\\
  $D_k$ & $k \geq 4$ & $x^2 y + y^{k-1}$ && $3$ & $k$ & $\frac{k}{2(k-1)}$\\
  \rowcolor{rowcolor}
  $E_6$ && $x^3 + y^4$ && $3$ & $6$ & $\frac{7}{12}$\\
  $E_7$ && $x^3 + x y^3$ && $3$ & $7$ & $\frac{5}{9}$\\
  \rowcolor{rowcolor}
  $E_8$ && $x^3 + y^5$ && $3$ & $8$ & $\frac{8}{15}$\\
  $T_{2,3,6}$ && $a x^2 y^2 + x^3 + y^6$ & $4a^3 + 27 \neq 0$ & $3$ & $10$ & $\frac{1}{2}$\\
  \rowcolor{rowcolor}
  $T_{2,4,4}$ && $a x^2 y^2 + x^4 + y^4$ & $a^2 \neq 4$ & $4$ & $9$ & $\frac{1}{2}$\\
  $T_{2,q,r}$ & $\frac{1}{q} + \frac{1}{r} < \frac{1}{2}$ & $a x^2 y^2 + x^q + y^r$ & $a \neq 0$ & $4$ & $q+r+1$ & $\frac{1}{2}$\\
  \rowcolor{rowcolor}
  $Z_{11}$ && $x^3 y + y^5 + a x y^4$ && $4$ & $11$ & $\frac{7}{15}$\\
  $Z_{12}$ && $x^3 y + x y^4 + a x^2 y^3$ && $4$ & $12$ & $\frac{5}{11}$\\
  \rowcolor{rowcolor}
  $W_{12}$ && $x^4 + y^5 + a x^2 y^3$ && $4$ & $12$ & $\frac{9}{20}$\\
  $W_{13}$ && $x^4 + x y^4 + a y^6$ && $4$ & $13$ & $\frac{7}{16}$\\
  \rowcolor{rowcolor}
  $N_{16}$ && $\begin{aligned}x^5 & + a x^3 y^2 + b x^2 y^3\\& + y^5 + c x^3 y^3\end{aligned}$ & $f_5$ square-free & $5$ & $16$ & $\frac{2}{5}$\\
  \bottomrulecolor
\end{tabular}
\end{table}

\begin{remark} \label{rem:multiple notations}
\begin{enumerate}[label=(\alph*), ref=\alph*]
\item We have added the polynomial for $N_{16}$ in Table~\ref{tab:normal form notation} which does not appear in \cite[Sections~15]{AGZV85}. The polynomial for $N_{16}$ defines a normal form by \cite[Theorem~3.20]{BMP20}. By \cite[Exercise~I.2.1.5]{GLS07}\footnote{There is an error in the exercise, it should say $\mu(f) = d(d-1) - k + 1$ instead of $\mu(f) = d(d-1) - k$.}, the $\mu$-constant stratum of $N_{16}$ is the open dense subset where the homogeneous degree $5$ part is a product of five pairwise coprime linear forms.
\item By \emph{$A_k$ singularity}, \emph{$D_k$ singularity}, \ldots, \emph{$N_{16}$ singularity}, we mean a complex space germ $(X, P)$ isomorphic to a complex space subgerm $(\mathbb V(f), \bm 0)$ of $(\mathbb C^n, \bm 0)$ where the stable right equivalence class of $f \in \mathbb C\{x_1, \ldots, x_n\}$ contains a polynomial which is in the $\mu$-constant stratum of respectively $A_k, D_k, \ldots, N_{16}$.
\item \label{itm:multiple notations} Normal forms are usually considered up to stable right equivalence, meaning that if $f$ and $g$ are stably right equivalent, then the normal forms of $f$ and $g$ are considered to be the same. Due to this, there are several different notations for some of the normal forms in Table~\ref{tab:normal form notation}:
\begin{enumerate}[label=(\arabic*), ref=\arabic*]
\item $T_{2, 3, 6+k} = J_{10+k} = J_{2, k}$ for all nonnegative integers~$k$,
\item $T_{2, 4, 4+k} = X_{9+k} = X_{1, k}$ for all nonnegative integers~$k$,
\item $T_{2, 4+r, 4+s} = Y_{4+r, 4+s} = Y_{r, s}^1$ for all positive integers $r$ and~$s$.
\end{enumerate}
\end{enumerate}
\end{remark}

\begin{lemma} \label{thm:lct's and Milnor numbers}
Let $f$ be one of the polynomials in the column \emph{normal form} in Table~\ref{tab:normal form notation}, satisfying the corresponding restrictions in column \emph{$\mu$-constant stratum}. Then $f$ has multiplicity~$\operatorname{mult}$ and Milnor number $\mu$ and $(\mathbb A^2, f)$ has log canonical threshold $\mathrm{lct}$ at the origin as given in Table~\ref{tab:normal form notation}.
\end{lemma}

\begin{proof}
The power series for $T_{2,q,r}$ are Newton non-degenerate and the power series for the other singularities in Table~\ref{tab:normal form notation} are semiquasihomogeneous. There are combinatorial formulas for the Milnor number in these cases, see \cite[Proposition~I.2.16 and Corollary~I.2.18]{GLS07}.

Choose the weights $(2, q-2)$ for $(x, y)$ and let $f$ be a power series for $T_{2,q,r}$ in Table~\ref{tab:normal form notation}. Since the pair $(\mathbb C^2,\, \frac12 \mathbb{V}(ax^2 y^2 + x^q))$ is log canonical outside the origin, by Lemma~\ref{lem:computing lct} the log canonical threshold of $f$ at the origin is~$\frac12$. The power series for all the other singularities in Table~\ref{tab:normal form notation} are semiquasihomogeneous and Lemma~\ref{lem:computing lct} gives a combinatorial formula for the log canonical threshold.
\end{proof}

\providecommand{\bysame}{\leavevmode\hbox to3em{\hrulefill}\thinspace}
\providecommand{\MR}{\relax\ifhmode\unskip\space\fi MR }
\providecommand{\MRhref}[2]{%
  \href{http://www.ams.org/mathscinet-getitem?mr=#1}{#2}
}
\providecommand{\href}[2]{#2}

\vspace{0.5\baselineskip}
\noindent\begin{minipage}{\linewidth}
  \noindent\begin{tabular}[t]{@{}l}
    Erik Paemurru\quad \texttt{algebraic.geometry@runbox.com}\\
  \end{tabular}\\
  Institute of Mathematics and Informatics, Bulgarian Academy of Sciences, Acad.\ G.~Bonchev Str.\ bl.~8, 1113, Sofia, Bulgaria.\\
  \textit{Former addresses:} Mathematik und Informatik, Universität des Saarlandes, 66123 Saarbrücken, Germany.\\
  Department of Mathematics, University of Miami, Coral Gables, Florida 33146, USA.%
\end{minipage}

\vspace{\baselineskip}%

\noindent\begin{minipage}{\linewidth}
  \noindent\begin{tabular}[t]{@{}l}
    Nivedita Viswanathan\quad \texttt{n.viswanathan@qmul.ac.uk}\\
  \end{tabular}\\
  School of Physical and Chemical Sciences, Queen Mary University of London, London, E1 4NS.\\
  \textit{Former Address:} Department of Mathematics, Brunel University London, Kingston Lane, Uxbridge, United Kingdom UB8 3PH.%
\end{minipage}

\end{document}